\documentclass[11pt]{amsart}%
\usepackage{amsfonts}
\usepackage{amsmath}
\usepackage{amssymb}
\usepackage{graphicx}%
\providecommand{\U}[1]{\protect\rule{.1in}{.1in}}

\theoremstyle{plain}

\newtheorem{corollary}{Corollary}

\newtheorem{definition}{Definition}
\newtheorem{example}{Example}

\newtheorem{lemma}{Lemma}

\newtheorem{proposition}{Proposition}
\newtheorem{remark}{Remark}

\numberwithin{equation}{section}
\begin{document}
\title[Riesz spaces with topologically full center]{Characterization of Riesz spaces with topologically full center}
\author{\c{S}afak Alpay}
\address{Mathematics Department, Middle East Technical University\\
06531 Ankara, Turkiye}
\email{safak@metu.edu.tr}
\author{Mehmet Orhon}
\address{Department of Mathematics and Statistics, University of New Hampshire\\
Durham, NH 03824, USA}
\email{mo@unh.edu}
\urladdr{}
\thanks{.}
\thanks{.}
\date{February 10, 2014}
\subjclass[2000]{Primary 47B38, 46B42; Secondary 47B60, 46H25}
\keywords{Riesz space, ideal center, orthomorphism, Arens extension}

\begin{abstract}
Let $E$ be a Riesz space and let $E^{\sim}$ denote its order dual. The
orthomorphisms $Orth(E)$ on $E,$ and the ideal center $Z(E)$ of $E,$ are
naturally embedded in $Orth(E^{\sim})$ and $Z(E^{\sim})$ respectively. We
construct two unital algebra and order continuous Riesz homomorphisms
\[
\gamma:((Orth(E))^{\sim})_{n}^{\sim}\rightarrow Orth(E^{\sim})\text{ }%
\]
and
\[
m:Z(E)^{\prime\prime}\rightarrow Z(E^{\sim})
\]
that extend the above mentioned natural inclusions respectively. Then, the
range of $\gamma$ is an order ideal in $Orth(E^{\sim})$ if and only if $m$ is
surjective. Furthermore, $m$ is surjective if and only if $E$ has a
topologically full center. (That is, the $\sigma(E,E^{\sim})$-closure of
$Z(E)x$ contains the order ideal generated by $x$ for each $x\in E_{+}.$) As a
consequence, $E$ has a topologically full center $Z(E)$ if and only if
$Z(E^{\sim})=\pi\cdot Z(E)^{\prime\prime}$ for some idempotent $\pi\in
Z(E)^{\prime\prime}.$

\end{abstract}
\maketitle

\section{Introduction}

Let $E$ be a Banach lattice and let $Z(E)$ be its (ideal ) center. In general
$Z(E)$ is a subalgebra and sublattice of $Z(E^{\prime})$, the center of the
Banach dual $E^{\prime}$ of $E.$ It is possible to extend this embedding to a
contractive algebra and lattice homomorphism of $Z(E)^{\prime\prime}$ into
$Z(E^{\prime}).$ It is natural to ask when the homomorphism would be onto
$Z(E^{\prime})$. It is clear that $Z(E)$ has to be large, since $Z(E^{\prime
})$ is always large. It turns out that the concept of largeness best suiting
the center $Z(E)$ in this problem is that $Z(E)$ should be topologically full
in the sense of Wickstead \cite{W1}. Namely, for each $x\in E_{+},$ the
closure of $Z(E)x$ is the closed ideal generated by $x.$ Then it is shown that
the homomorphism is onto $Z(E^{\prime})$ if and only if $Z(E)$ is
topologically full \cite[Corollary 2]{O1}.

Our purpose in this $\ $paper is to consider the corresponding problem for a
Riesz space $E$ with point separating order dual $E^{\sim}$. In the case of
Riesz spaces however, there are, in general, two different algebras that act
on $E$ which should be considered. Namely, $Orth(E),$ the algebra of the
orthomorphisms on $E,$ and its subalgebra and (order) ideal $Z(E),$ the center
of $E.$ In the case of Banach lattices $Orth(E)=Z(E),$ therefore this problem
does not arise. Recall that an orthomorphism on $E$ is an order bounded
operator on $E$ that preserves bands, and $Z(E)$ is the ideal generated by the
identity operator on $E$ in the Riesz space $Orth(E).$ In fact $Orth(E)$ is an
$f$-algebra. The ideal center $Z(E)$, on the other hand, is a normed
$AM$-lattice where the identity operator is the order unit of the
$AM$-lattice. Similar to the Banach lattice case, $Orth(E)$ is embedded in
$Orth(E^{\sim})$ as an $f$-subalgebra and sublattice. Also as before, $Z(E)$
is embedded in $Z(E^{\sim})$ as a subalgebra and sublattice. We construct two
unital algebra and order continuous lattice homomorphisms
\[
\gamma:(Orth(E)^{\sim})_{n}^{\sim}\rightarrow Orth(E^{\sim})
\]
and%
\[
m:Z(E)^{\prime\prime}\rightarrow Z(E^{\sim})
\]
that extend the two embeddings mentioned above respectively. Then the range of
$\gamma$ is an order ideal in $Orth(E^{\sim})$ if and only if $m$ is onto
$Z(E^{\sim})$ (Corollary \ref{c2}). Also, $m$ is onto $Z(E^{\sim})$ if and
only if $E$ has a topologically full center (Proposition \ref{p5}). That is
for each $x\in E_{+}$, the $\sigma(E,E^{\sim})$-closure of $Z(E)x$ in $E$
contains the ideal generated by $x.$ It follows that $E$ has a topologically
full center $Z(E)$ if and only if the ideal center of its order dual $E^{\sim
}$ is given as $Z(E^{\sim})=\pi\cdot Z(E)^{\prime\prime}$ for some idempotent
$\pi\in Z(E)^{\prime\prime}.$

We point out that the proofs of the above mentioned results differ from those
used in the Banach lattice case \cite{O1}. The method used in this paper owes
a lot to the work and results of Huijsmans and de Pagter \cite{HdP1} on the
bidual of an $f$-algebra.

In the Banach lattice case it is shown that if $Z(E)$ is topologically full
then it is maximal abelian \cite[Corollary 3]{O1}. Wickstead \cite{W2} showed
that the converse is not true. He constructed an interesting $AM$-lattice with
a center that is maximal abelian but is not topologically full. In the case of
Riesz spaces, even when $Z(E)$ is topologically full, it need not be maximal
abelian. In fact, when $Z(E)$ is topologically full, its commutant in the
order bounded operators is $Orth(E)$ (Corollary 4).

If $E$ is a Riesz space, by $E^{\sim}$ we will denote the Riesz space of order
bounded linear functionals on $E$. All Riesz spaces considered in this paper are assumed to have separating order duals. $E_{n}^{\sim}$ will
denote the order continuous linear functionals in $E^{\sim}$. We let
$L_{b}(E,F)$ denote the space of order bounded linear operators from the Riesz
space $E$ into the Riesz space $F.$ When $T:E\rightarrow F$ is an order
bounded operator between two Riesz spaces, the adjoint of $T$ carries
$F^{\sim}$ into $E^{\sim}$ and we will denote it by $T^{\prime}.$ The dual of
a normed space will be denoted by $E^{\prime}.$ In all terminology concerning
Riesz spaces we will adhere to the definitions in \cite{AB} and \cite{Z}.

Let us recall that for any associative algebra $A,$ a multiplication (called
the Arens multiplication) can be introduced in the second algebraic dual
$A^{\ast\ast}$ of $A$ \cite{Ar}. This is accomplished in three steps: given
$a,b\in A,$ $f\in A^{\ast}$ and $F,G\in A^{\ast\ast}$, one defines $f\cdot
a\in A^{\ast}$, $F\cdot f\in A^{\ast}$, and $F\cdot G\in A^{\ast\ast}$ by the
equations%
\begin{align*}
(f\cdot a)(b)  &  =f(ab)\\
(F\cdot f)(a)  &  =F(f\cdot a)\\
(F\cdot G)(f)  &  =F(G\cdot f)
\end{align*}
For any Archimedean $f$-algebra $A,$ the space $(A^{\sim})_{n}^{\sim}$ is an
Archimedean $f$-algebra with respect to the Arens multiplication \cite{HdP1}.

We denote for any $F\in(A^{\sim})_{n}^{\sim}$ the mapping $f\rightarrow F\cdot
f$ by $\nu_{F}.$ The map $\nu_{F}$ is an orthomorphism on $A^{\sim}%
$\cite{HdP1}. The mapping $\nu:(A^{\sim})_{n}^{\sim}\rightarrow Orth(A^{\sim
}),$ defined by $\nu(F)=\nu_{F}$ is an algebra and Riesz homomorphism for any
Archimedean $f$-algebra $A.$ Moreover $\nu$ is onto $Orth(A^{\sim})$ if and
only if $(A^{\sim})_{n}^{\sim}$ has a unit element. In that case, $\upsilon$
is injective by Theorem 5.2 in \cite{HdP1}. The main purpose of this $\ $paper
is to extend the latter result to an arbitrary Riesz space.

\section{The Arens Homomorphism}

Let $E$ be a Riesz space and consider the bilinear map%
\begin{equation}
\ Orth(E)\,\times E\rightarrow E \tag{1}\label{eq1}%
\end{equation}
defined by $(\pi,x)\rightarrow\pi(x)$ for each $\pi\in\ Orth(E)$%
\thinspace\ and $x\in E.$ Related to (\ref{eq1}), we define the following
bilinear maps:%
\begin{equation}
E\times E^{\sim}\rightarrow(\ Orth(E)\,)^{\sim}::(x,f)\rightarrow\psi
_{x,f}:\psi_{x,f}(\pi)=f(\pi x) \tag{2}\label{eq2}%
\end{equation}%
\begin{equation}
E^{\sim}\times(\ Orth(E)\,)^{\sim\sim}\rightarrow E^{\sim}::(f,F)\rightarrow
F\bullet f:F\bullet f(x)=F(\psi_{x,f}) \tag{3}\label{eq3}%
\end{equation}
where $x\in E,$ $f\in E^{\sim},$ $\pi\in\ Orth(E)$\thinspace\ and
$F\in\ Orth(E)$\thinspace$)^{\sim\sim}.$ We call the map defined in
(\ref{eq3}) the\emph{ Arens extension} of the map in (\ref{eq1}).

For an Archimedean unital $f$-algebra $A$, we have $A^{\sim\sim}=(A^{\sim
})_{n}^{\sim}$ by Corollary 3.4 in \cite{HdP1}. Since $\ Orth(E)$%
\thinspace\ is a unital $f$-algebra with point separating order dual, it is
Archimedean. This enables us to use the identification $(\ Orth(E)$%
\thinspace$)^{\sim\sim}=((\ Orth(E)$\thinspace$)^{\sim})_{n}^{\sim}$
throughout this $\ $paper. It is straightforward to check that the Arens
product on the $f$-algebra $((\ Orth(E)$\thinspace$)^{\sim})_{n}^{\sim}$ is
compatible with the Arens extension defined in (\ref{eq3}), that is, $E^{\sim
}$ is a unital module over $((\ Orth(E)$\thinspace$)^{\sim})_{n}^{\sim}.$

We use (\ref{eq3}) to define a linear operator%
\[
\gamma:((\ Orth(E)\,)^{\sim})_{n}^{\sim}\rightarrow L_{b}(E^{\sim})\text{ by
}\gamma(F)(f)=F\bullet f
\]
for each $F\in((\ Orth(E)$\thinspace$)^{\sim})_{n}^{\sim}$ and $f\in E^{\sim
}.$ We will call $\gamma$ the Arens homomorphism of the order bidual of
$\ Orth(E)$\thinspace$.$

\begin{proposition}
\label{p1} $\gamma$ is a unital algebra and order continuous Riesz
homomorphism such that $\gamma\left(  ((\ Orth(E)\,)^{\sim})_{n}^{\sim
}\right)  \subset Orth(E^{\sim}).$
\end{proposition}

\begin{proof}
It follows from the definition of $\gamma$ that \ $\gamma$ is a positive order
continuous unital algebra homomorphism. Also it is easily checked that
$\gamma(\pi)=\pi^{\prime}$ for each $\pi\in\ Orth(E)$\thinspace$.$ Let
$F\in((\ Orth(E)$\thinspace$)^{\sim})_{n}^{\sim}$ such that $|F|\leq\pi$ for
some $\pi\in\ Orth(E)$\thinspace$.$ Then $|\gamma(F)|\leq\gamma(|F|)\leq
\gamma(\pi)=\pi^{\prime}$ by the positivity of $\gamma.$ Since $\pi^{\prime
}\in Orth(E^{\sim}),$ we see that $\gamma(F)\in Orth(E^{\sim}).$ Then, since
$\gamma$ is order continuous and the ideal generated by $\ Orth(E)$%
\thinspace\ is strongly order dense in $((\ Orth(E)$\thinspace$)^{\sim})_{n}^{\sim}%
$,\ the range of $\gamma$ is contained in $Orth(E^{\sim}).$ That $\gamma$ is a
Riesz homomorphism follows from the fact that it is an algebra homomorphism by
Corollary 5.5 in \cite{HdP2}.
\end{proof}

If $A$ is an $f$-algebra then the range of the homomorphism $\nu:(A^{\sim
})_{n}^{\sim}\rightarrow Orth(A^{\sim})$ is contained in the range of the
Arens homomorphism $\gamma.$

\begin{proposition}
\label{p2} Let $A$ be an $f$-algebra with point separating order dual. The
range of the homomorphism $\nu:(A^{\sim})_{n}^{\sim}\rightarrow Orth(A^{\sim
})$ is contained in the range of the Arens homomorphism $\gamma.$
\end{proposition}

\begin{proof}
Let $P:A\rightarrow Orth(A)$ be the canonical embedding of $A$ into
$\ Orth(A)$\thinspace\ (i.e., $P(a)(b)=ab$ for all $a,b\in A$). It is well
known that $P(A)$ is a sublattice and an algebra ideal in $Orth(A).$ For each
$\mu\in(\ Orth(A)$\thinspace$)^{\sim}$, define $\widehat{\mu}\in A^{\ast}$ by
$\widehat{\mu}(a)=\mu(P(a)).$ Since positivity is preserved, it is clear that
$\widehat{\mu}\in A^{\sim}$ for each $\mu\in(\ Orth(A)$\thinspace$)^{\sim}.$
Given $a\in A,$ $f\in A^{\sim}$, we have $\psi_{a,f}\in(\ Orth(A)$%
\thinspace$)^{\sim}.$ Then%
\[
\widehat{\psi}_{a,f}(b)=\psi_{a,f}(P(b))=f(P(b)(a))=f(ab)=(f\cdot a)(b)
\]
for each $b\in A.$ That is $\widehat{\psi}_{a,f}=f\cdot a$ for all $a\in A,$
$f\in A^{\sim}.$ Now let $F\in(A^{\sim})_{n}^{\sim}.$ Define $\widehat{F}%
\in((\ Orth(A)$\thinspace$)^{\sim})^{\ast}$ by $\widehat{F}(\mu)=F(\widehat
{\mu})$ for each $\mu\in(\ Orth(A)$\thinspace$)^{\sim}.$ Since $0\leq\mu$
implies $0\leq\widehat{\mu},$ we have $0\leq\widehat{F}$ whenever $0\leq F.$
That is $\widehat{F}\in((\ Orth(A)$\thinspace$)^{\sim})^{\sim}=((\ Orth(A)$%
\thinspace$)^{\sim})_{n}^{\sim}.$ For each $F\in(A^{\sim})_{n}^{\sim},$ $f\in
A^{\sim}$ and $a\in A,$ we have%
\begin{align*}
\nu_{F}(f)(a)  &  =F\cdot f(a)=F(f\cdot a)=F(\widehat{\psi}_{a,f})=\widehat
{F}(\psi_{a,f})\\
&  =\widehat{F}\bullet f(a)=\gamma(\widehat{F})(f)(a).
\end{align*}
Hence $\nu_{F}=\gamma(\widehat{F}).$
\end{proof}

Let us check the behavior of $\gamma$ in some specific examples.

\begin{example}
\label{ex1} Let $\omega$ denote all sequences and let $A=l^{1}$ the
$f$-subalgebra of $\omega$ consisting of absolutely summable sequences. Then
$A^{\sim}=l^{\infty}$ and $\ Orth(A)$\thinspace$=Orth(A^{\sim})=l^{\infty}.$
Since $(A^{\sim})_{n}^{\sim}=l^{1},$ $\nu$ is the inclusion map $l^{1}%
\rightarrow l^{\infty}$ so that $\nu$ is one-to-one and not onto. On the other
hand $((\ Orth(A)$\thinspace$)^{\sim})_{n}^{\sim}=(l^{\infty})^{\prime\prime}$
and $\gamma$ is the band projection of $(l^{\infty})^{\prime\prime}$ onto
$l^{\infty}.$ Thus $\gamma$ is onto and not one-to-one.
\end{example}

\begin{example}
\label{ex2} Let $A=c_{0}$ be the $f$-subalgebra of $\omega$ consisting of the
sequences convergent to zero. Then $A^{\sim}=l^{1}$ and $\ Orth(A)$%
\thinspace$=Orth(A^{\sim})=l^{\infty}.$ Since $(A^{\sim})_{n}^{\sim}%
=l^{\infty},$ $\nu$ is the identity map on $l^{\infty}$ so that $\nu$ is
one-to-one and onto. On the other hand $((\ Orth(A)$\thinspace$)^{\sim}%
)_{n}^{\sim}=(l^{\infty})^{\prime\prime}$ and $\gamma$ is the band projection
of $(l^{\infty})^{\prime\prime}$ onto $l^{\infty}.$ Thus $\gamma$ is onto and
not one-to-one.
\end{example}

\begin{example}
\label{ex3} Consider $C[0,1]$ with the product $\ast$ defined by $a\ast b=iab$
with $i(x)=x$ for all $x\in\lbrack0,1].$ Then $A=(C[0,1],\ast)$ is an
Archimedean $f$-algebra. As shown in \cite{HdP1}, $(A^{\sim})_{n}^{\sim}$ is
not semi-prime. Therefore $\upsilon$ is not one-to-one and not onto
\cite{HdP1}. On the other hand $\ Orth(A)$\thinspace$=Z(C[0,1])=C[0,1]$ and
$Orth(A^{\sim})=Z(C[0,1]^{\prime})=C[0,1]^{\prime\prime}.$ Since
$(C[0,1]^{\sim})_{n}^{\sim}=C[0,1]^{\prime\prime},$ $\gamma$ is the identity
map on $C[0,1]^{\prime\prime}.$ Therefore $\gamma$ is one-to-one and onto.
\end{example}

\section{Riesz spaces with topologically full center}

We start with the following definition that is due to Wickstead \cite{W1} in
the case of Banach lattices.

\begin{definition}
\label{df1} Suppose $E$ is a Riesz space. Then $E$ is said to have a
topologically full center if for each $x\in E_{+}$ the $\sigma(E,E^{\sim}%
)$-closure of $Z(E)x$ contains the ideal generated by $x.$
\end{definition}

Banach lattices with topologically full center were initiated in \cite{W1}.
The class of Riesz spaces and the class Banach lattices with topologically
full center are quite large. For example, in a $\sigma$-Dedekind complete
Riesz space $E$ each positive element generates a projection band. Therefore
for each $x\in E_{+}$, $Z(E)x$ is an ideal and $Z(E)$ is topologically full.
Also Banach lattices with a quasi-interior point or with a topological
orthogonal system have topologically full center \cite{W2}. However not all
Riesz spaces have topologically full center.

\begin{example}
\label{ex4} (Zaanen \cite[p.664]{Z}) Let $E$ be the Riesz space of piecewise
affine continuous functions on $[0,1]$. Clearly the ideal generated by the
constant $1$ function equals $E.$ But, as shown by Zaanen, $Z(E)$ is trivial,
that is, it consists of the scalar multiples of the identity. Therefore $E$
does not have a topologically full center.
\end{example}

The first example of an AM-space that has trivial center was given in
\cite{G}. A thorough study of Banach lattices with trivial center was
undertaken in \cite{W3}. We refer the reader to \cite{W3} for further examples
of Banach lattices with trivial center as well as a careful treatment of the
following example of Goullet de Rugy mentioned above.

\begin{example}
\label{ex5}(\cite{G}) Let $K$ be a compact Hausdorff space with a point $p\in
K$ such that $\{p\}$ is not a $G_{\delta}$-set in $K$ (e.g., \cite[Example 4,
p.140]{S}). Let $C_{0}(K)$ denote the elements of $C(K)$ that vanish at $p.$
Let $H$ denote the positive unit ball of $C_{0}(K)^{\prime}$ with the relative
$\sigma(C_{0}(K)^{\prime},C_{0}(K))$-topology. Let $E=\{f\in C(H):f(r\mu
)=rf(\mu)$ for all $r\in\lbrack0,1]$ and for each $\mu\in H$ with $\Vert
\mu\Vert=1\}$. Then $E$ is an AM-space (without order unit). As a sublattice
of $l^{\infty}(H\smallsetminus\{0\}),$ one has $E^{d}=\{0\}.$ Therefore $Z(E)$
is embedded in $l^{\infty}(H\smallsetminus\{0\})$ \cite{W4}. Then one may
compute that $Z(E)$ consists of continuous bounded functions on
$H\smallsetminus\{0\}$ that are constant on the rays (i.e., $g(r\mu)=g(\mu)$
for all $r\in(0,1],$ for each $\mu\in H$ with $\left\Vert \mu\right\Vert =1$).
It follows by an argument in \cite[p.371]{G} that if there is a non-constant
function in $Z(E)$ then $\{0\}$ is a $G_{\delta}$-set in $H$. That in turn
implies that $\{p\}$ would be a $G_{\delta}$-set in $K.$ Therefore $Z(E)$ is
trivial and $E$ does not have a topologically full center.
\end{example}

$Z(E)$ is an Archimedean unital $f$-algebra with order unit. The order unit
norm induced on $Z(E)$ is an algebra and lattice norm. $\widehat{Z(E)}$, the
norm completion of $Z(E),$ is an AM-space and a partially ordered Banach
algebra where the order unit and the algebra unit coincide. Therefore by the
Stone algebra theorem $\widehat{Z(E)}\cong C(K)$ (isometric algebra and
lattice homomorphism) for some compact Hausdorff space $K.$ (Here $C(K)$
denotes the real valued continuous functions on $K.$) Then $Z(E)^{\prime
}=\widehat{Z(E)}^{\prime}=\widehat{Z(E)}^{\sim}=Z(E)^{\sim}$ and
$Z(E)^{\sim\sim}=Z(E)^{\prime\prime}.$ $Z(E)^{\prime\prime}$ is an AM-space
and with the Arens product, it is a partially ordered Banach algebra (an
Archimedean $f$-algebra with unit) where the order unit and the algebra unit
coincide. Therefore $Z(E)^{\prime\prime}\cong C(S)$ for some hyperstonian
space $S.$ That is the Arens product on $Z(E)^{\prime\prime}$ coincides with
the pointwise product on $C(S)$ \cite{Ar}.

Given the bilinear map
\begin{equation}
Z(E)\times E\rightarrow E \tag{4}\label{eq4}%
\end{equation}
defined by $(T,x)\rightarrow Tx$ for each $T\in Z(E)$ and $x\in E,$ we define
the following bilinear maps:%
\begin{equation}
E\times E^{\sim}\rightarrow Z(E)^{\prime}::(x,f)\rightarrow\mu_{x.f}%
=\psi_{x,f}|_{Z(E)} \tag{5}\label{eq5}%
\end{equation}%
\begin{equation}
E^{\sim}\times Z(E)^{\prime\prime}\rightarrow E^{\sim}::(f,F)\rightarrow
F\circ f:F\circ f(x)=F(\mu_{x,f}) \tag{6}\label{eq6}%
\end{equation}
where $x\in E,$ $f\in E^{\sim}$ and $F\in Z(E)^{\prime\prime}.$ The Arens
product on $Z(E)^{\prime\prime}$ is compatible with the bilinear map defined
in (\ref{eq6}). That is, $E^{\sim}$ is a unital module over $Z(E)^{\prime
\prime}.$ (\ref{eq6}) allows us to define a linear operator $m:Z(E)^{\prime
\prime}\rightarrow L_{b}(E^{\sim})$ where $m(F)(f)=F\circ f$ for all $f\in
E^{\sim}$ and $F\in Z(E)^{\prime\prime}.$ It is easily checked that
$m(T)=\gamma(T)$ whenever $T\in Z(E).$

We have the following analogue of Proposition \ref{p1} for the map $m$.

\begin{proposition}
\label{p3} $m$ is a unital algebra and order continuous lattice homomorphism
such that $m\left(  Z(E)^{\prime\prime}\right)  \subset Z(E^{\sim}).$
\end{proposition}

\begin{proof}
That $m$ is a positive order continuous algebra homomorphism is immediate from
the definition of $m$. That it is a lattice homomorphism follows as in
Proposition \ref{p1}.

For each $F\in Z(E)^{\prime\prime},$ there is a net $\{T_{\alpha}\}$ in $Z(E)$
such that $\parallel T_{\alpha}\parallel\leq\parallel F\parallel$ and
$T_{\alpha}\rightarrow F$ in $\sigma(Z(E)^{\prime\prime},Z(E)^{\prime}%
)$-topology. Let $f\in E_{+}^{\sim}$ and $x\in E_{+}$. Then
\[
-||F||f(x)\leq F\circ f(x)=\underset{\alpha}{\lim}f(T_{\alpha}x)\leq
||F||f(x).
\]
So $F\in Z(E^{\sim}).$
\end{proof}

We will call the map $m:Z(E)^{\prime\prime}\rightarrow Z(E^{\sim})$ the Arens
homomorphism of the bidual of $Z(E)$ (into $Z(E^{\sim})$).

\begin{proposition}
\label{p4} Let $A$ be a unital $f$-algebra with point separating order dual.
Then the Arens homomorphism of the bidual of $Z(A)$ is onto $Z(A^{\sim}).$
\end{proposition}

\begin{proof}
Let $F\in Z(A^{\sim})$ with $0\leq F\leq1.$ Since $A$ is unital, Theorem 5.2
in \cite{HdP1} implies that the algebra and lattice homomorphism $\nu:(A^{\sim
})_{n}^{\sim}\rightarrow Orth(A^{\sim})$ of \cite{HdP1} is one-to-one and
onto. Therefore with a slight abuse of notation we will identify $(A^{\sim
})_{n}^{\sim}$ with $Orth(A^{\sim})$. In the duality $\langle(A^{\sim}%
)_{n}^{\sim},A^{\sim}\rangle$, $A$ is dense in $(A^{\sim})_{n}^{\sim}$ with
respect to the weak topology. Since the locally solid convex topology
$|\sigma|((A^{\sim})_{n}^{\sim},A^{\sim})$ is compatible with the duality
$\langle(A^{\sim})_{n}^{\sim},A^{\sim}\rangle$ \cite[Theorem 11.13, p.170]%
{AB}, there is a net $\{a_{\alpha}\}$ in $A$ such that $\{a_{\alpha}\}$
converges to $F$ in the $|\sigma|((A^{\sim})_{n}^{\sim},A^{\sim})$-topology.
Lattice operations are continuous in the locally solid convex $|\sigma
|((A^{\sim})_{n}^{\sim},A^{\sim})$-topology and $A$ is a sublattice of
$(A^{\sim})_{n}^{\sim}$ with $1\in A.$ Therefore we may suppose that $0\leq
a_{\alpha}\leq1.$ Hence, by the Alaoglu Theorem, there is $T\in Z(A)^{\prime
\prime}$ with $0\leq T\leq1$ such that a subnet $\{a_{\beta}\}$ converges to
$T$ in the $\sigma(Z(A)^{\prime\prime},Z(A)^{\prime})$-topology. Then, when
$f\in A^{\sim}$ and $b\in A,$%
\[
T\circ f(b)=T(\mu_{b,f})=\underset{\beta}{\lim}\mu_{b,f}(a_{\beta}%
)=\underset{\beta}{\lim}\psi_{b,f}(a_{\beta})=\underset{\beta}{\lim}%
f(a_{\beta}b).
\]
On the other hand, since $\{a_{\beta}\}$ converges to $F$ in the
$|\sigma|((A^{\sim})_{n}^{\sim},A^{\sim})$-topology implies convergence also
in the $\sigma((A^{\sim})_{n}^{\sim},A^{\sim})$-topology, we have%
\[
F\cdot f(b)=F(f\cdot b)=\underset{\beta}{\lim}f\cdot b(a_{\beta}%
)=\underset{\beta}{\lim}f(a_{\beta}b).
\]
Therefore $m(T)=F.$
\end{proof}

Note that if $E$ is a Riesz space and if we set $A=Orth(E)$, then $Z(A)=Z(E).$
Also, since then $A$ is a unital $f$-algebra with point separating order dual,
Proposition \ref{p4} will be true for $A.$ In what follows we will use these
facts repeatedly.

\begin{corollary}
\label{c1} Let $E$ be a Riesz space. Let $A=Orth(E)$ and $m_{A}$ denote the
Arens homomorphism of the bidual of $Z(A)$ onto $Z(A^{\sim})$ (where
$Z(A)=Z(E)$). Then\newline(1) the following diagram is commutative%
\[%
\begin{array}
[c]{ccccc}%
Z(E)^{\prime\prime} & \overset{m}{\rightarrow} & Z(E^{\sim}) & \overset
{i}{\rightarrow} & Orth(E^{\sim})\\
\shortparallel &  &  &  & \uparrow\gamma\\
Z(A)^{\prime\prime} & \overset{m_{A}}{\rightarrow} & Z(A^{\sim}) & \overset
{i}{\rightarrow} & (A^{\sim})_{n}^{\sim}%
\end{array}
\]
where $i$ denotes the natural inclusion map;\newline(2) $\gamma(Z((A^{\sim
})_{n}^{\sim}))=m(Z(E)^{\prime\prime}).$
\end{corollary}

\begin{proof}
(1) Suppose $F\in Z(A)^{\prime\prime}$ with $0\leq F\leq1.$ Then, as in the
proof of Proposition \ref{p4}, let $0\leq a_{\alpha}\leq1$ be a net in
$Z(A)\subset A$ such that $\{a_{\alpha}\}$ converges to $F$ in the
$\sigma(Z(A)^{\prime\prime},Z(A)^{\prime})$-topology and also to $m_{A}%
(F)\in(A^{\sim})_{n}^{\sim}$ in the $\sigma((A^{\sim})_{n}^{\sim},A^{\sim}%
)$-topology. Since $(A^{\sim})_{n}^{\sim}=(Orth(E)^{\sim})_{n}^{\sim},$ for
each $x\in E$ and $f\in E^{\sim},$ we have
\begin{align*}
\gamma(m_{A}(F))(f)(x)  &  =m_{A}(F)\bullet f(x)=m_{A}(F)(\psi_{x,f}%
)=\underset{\alpha}{\lim}\psi_{x,f}(a_{\alpha})\\
&  =\underset{\alpha}{\lim}\mu_{x,f}(a_{\alpha})=F(\mu_{x,f})=m(F)(f)(x).
\end{align*}

(2) As in proof of Proposition \ref{p4}, we may identify $(A^{\sim})_{n}%
^{\sim}$ with $Orth(A^{\sim})$ via the homomorphism $\nu$. Then $Z(A^{\sim
})=Z((A^{\sim})_{n}^{\sim}).$ Since, by Proposition \ref{p4}, $m_{A}$ is onto
$Z(A^{\sim}),$ from part (1), we have $\gamma(Z((A^{\sim})_{n}^{\sim
})=m(Z(E)^{\prime\prime}).$
\end{proof}

\begin{corollary}
\label{c2} Let $E$ be a Riesz space. Then $\gamma((Orth(E)^{\sim})_{n}^{\sim
})$ is an order ideal in $Orth(E^{\sim})$ if and only if $m(Z(E)^{\prime
\prime})=Z(E^{\sim}).$
\end{corollary}

\begin{proof}
Let $A=Orth(E).$ Suppose $\gamma((A^{\sim})_{n}^{\sim})$ is an order ideal in
$Orth(E^{\sim}).$ Since $\gamma$ is an order continuous lattice homomorphism,
the kernel $Ker(\gamma)$ is a band in $(A^{\sim})_{n}^{\sim}.$ Let $(1-e)$ be
the band projection of $(A^{\sim})_{n}^{\sim}$ onto $Ker(\gamma).$ Then
$e\cdot(A^{\sim})_{n}^{\sim}$ is algebra and lattice isomorphic to the order
ideal $\gamma((A^{\sim})_{n}^{\sim})$ with $\gamma(e)=1.$ Let $T\in Z(E^{\sim
})$ with $0\leq T\leq1=\gamma(e).$ So there is $F\in(A^{\sim})_{n}^{\sim}$
such that $\gamma(F)=T.$ Moreover, since $\gamma$ is an algebra and lattice
homomorphism, we may choose $F$ so that $0\leq F\leq e\leq1.$ That is we may
suppose that $F\in Z((A^{\sim})_{n}^{\sim}).$ Then Corollary \ref{c1} part (2)
implies there is $G\in Z(E)^{\prime\prime}$ such that $m(G)=\gamma(F)=T.$

Conversely, suppose $m(Z(E)^{\prime\prime})=Z(E^{\sim}).$ Suppose $0\leq
T\leq\gamma(F)$ for some $0\leq F\in(A^{\sim})_{n}^{\sim}$ and for some $T\in
Orth(E^{\sim}).$ Since $Orth(E^{\sim})$ is Dedekind complete, there is
$\widehat{T}\in Z(Orth(E^{\sim}))=Z(E^{\sim})$ such that $0\leq\widehat{T}%
\leq1$ and $T=\widehat{T}\gamma(F).$ By Corollary \ref{c1} part (2), there is
$G\in Z((A^{\sim})_{n}^{\sim})$ such that $\gamma(G)=\widehat{T}$ . Then
$T=\gamma(G)\gamma(F)=\gamma(G\cdot F)$ where $G\cdot F\in(A^{\sim})_{n}%
^{\sim}.$
\end{proof}

Corollary \ref{c2} gives the first part of the result stated in the abstract.
We will complete the result by showing that the only Riesz spaces that satisfy
Corollary \ref{c2} are necessarily those with a topologically full center. For
Banach lattices a proof of this was given in \cite{O1}. Initially we will give
the proof of the sufficiency.

\begin{proposition}
\label{p5} Let $E$ be a Riesz space. Then $E$ has a topologically full center
if and only if the Arens homomorphism $m:Z(E)^{\prime\prime}\rightarrow
Z(E^{\sim})$ is surjective. Then there exits an idempotent $\pi\in
Z(E)^{\prime\prime}$ such that $Z(E^{\sim})=\pi\cdot Z(E)^{\prime\prime}$ and
$Ker(m)=(1-\pi)\cdot Z(E)^{\prime\prime}.$
\end{proposition}

\begin{proof}
(Sufficiency) Suppose $m$ is surjective. To show that $E$ has topologically
full center it is sufficient to show that each $\sigma(E,E^{\sim})$-closed
$Z(E)$ submodule of $E$ is an ideal. This is equivalent to showing that each
$\sigma(E^{\sim},E)$-closed $Z(E)$ submodule of $E^{\sim}$ is an ideal. Let
$M$ be a $\sigma(E^{\sim},E)$-closed $Z(E)$ submodule of $E^{\sim}$ and let
$T\in Z(E)^{\prime\prime}.$ There is a net $\{T_{\alpha}\}$ in $Z(E)$ that
converges to $T$ in the $\sigma(Z(E)^{\prime\prime},Z(E)^{\prime})$-topology.
For each $x\in E$ and $f\in M$ we have%
\[
T_{\alpha}\circ f(x)=\mu_{x,f}(T_{\alpha})\rightarrow T(\mu_{x,f})=T\circ
f(x).
\]
Hence $M$ is a $Z(E)^{^{\prime\prime}}$-submodule of $E^{\sim}.$ Since
$Z(E^{\sim})=m(Z(E)^{\prime\prime}),$ $M$ is a $Z(E^{\sim})$-submodule.
$E^{\sim}$ is Dedekind complete. It is well known that a subspace of $E^{\sim
}$ is an ideal in $E^{\sim}$ if and only if it is a $Z(E^{\sim})$-submodule of
$E^{\sim}$ (e.g., \cite{W2}). Hence $M$ is an ideal in $E^{\sim}.$ Since $m$
is order continuous (Proposition \ref{p3}), $Ker(m)$ is a band in
$Z(E)^{\prime\prime}.$ Hence there exists a band projection $\pi\in
Z(E)^{\prime\prime}$ such that $Ker(m)=(1-\pi)\cdot Z(E)^{\prime\prime}$ and
$Z(E^{\sim})=\pi\cdot Z(E)^{\prime\prime}.$
\end{proof}

The proof of the converse requires some preparatory results.

Given a Riesz space $E$, let $A$ be a unital subalgebra of $Z(E).$ Let $A^{o}$
denote the polar of $A$ in $Z(E)^{\prime}$ and let $A^{oo}$ denote the polar
of $A^{o}$ in $Z(E)^{\prime\prime}.$ By standard duality theory, we have that
$A^{\prime}=Z(E)^{\prime}/A^{0}$ and $A^{\prime\prime}=A^{00}\subset
Z(E)^{\prime\prime}.$ Since $A$ is a normed algebra, $A^{\prime\prime}$ is a
Banach algebra with the Arens product \cite{Ar}. In fact $A^{\prime\prime}$ is
a subalgebra of $Z(E)^{\prime\prime}$ when $Z(E)^{\prime\prime}$ has its Arens product.

\begin{lemma}
\label{L1} Let $A$ be a unital subalgebra of $Z(E).$ Then $A^{00}$ is a
subalgebra of $Z(E)^{\prime\prime}$ and the algebra product on $A^{00}$ is
identical with the Arens product on $A^{\prime\prime}$ under the canonical
isomorphism of $A^{00}$ with $A^{\prime\prime}.$
\end{lemma}

\begin{proof}
Let $f\in A^{0}$ and $a\in A.$ It is easily checked that $f\cdot a\in A^{0}.$
Let $F,G\in A^{00}.$ There is a net $\{a_{\alpha}\}$ in $A$ that converges to
$F$ in the $\sigma(Z(E)^{\prime\prime},Z(E)^{\prime})$-topology. Let $f\in
A^{0}.$ Then%
\[
F\cdot G(f)=F(G\cdot f)=\underset{\alpha}{\lim}G\cdot f(a_{\alpha}%
)=\underset{\alpha}{\lim}G(f\cdot a_{\alpha})=0
\]
and $F\cdot G\in A^{00}.$

On the other hand, denote the isomorphism of $A^{\prime\prime}$ onto $A^{00}$
by $F\rightarrow\widehat{F}.$ Given $f\in A^{\prime}$, let $\widehat{f}\in
Z(E)^{\prime}$ denote any extension of $f$ on $Z(E).$ When $F,G\in
A^{\prime\prime},$ let $\{a_{\alpha}\}$ and $\{b_{\beta}\}$ be nets in $A$
that converge to $\widehat{F}$ and $\widehat{G}$ respectively in the
$\sigma(Z(E)^{\prime\prime},Z(E)^{\prime})$-topology. Then, evidently, the
respective nets also converge to $F$ and $G$ respectively in the
$\sigma(A^{\prime\prime},A^{\prime})$-topology. For any $f\in A^{\prime},$ we
have
\[
F\cdot G(f)=\underset{\alpha}{\lim}\underset{\beta}{\lim}f(a_{\alpha}b_{\beta
})=\underset{\alpha}{\lim}\underset{\beta}{\lim}\widehat{f}(a_{\alpha}%
b_{\beta})=\widehat{F}\cdot\widehat{G}(\widehat{f}).
\]

\end{proof}

\begin{lemma}
\label{L2} Let $I$ be a closed algebra ideal in $Z(E)$ and consider $A=Z(E)/I$
with the quotient norm. Then $A$ is a normed algebra and with the Arens
product $A^{\prime\prime}$ may be identified with the subalgebra of
$Z(E)^{\prime\prime}$ given by
\[
(I^{00})^{d}=\{F\in Z(E)^{\prime\prime}:|F|\wedge|G|=0\text{ for all }G\in
I^{00}\}.
\]

\begin{proof}
Let $\widehat{Z(E)}=C(K^{\prime})$ for some compact Hausdorff space
$K^{\prime}$ and let $\overline{I}$ denote the closure of $I$ in $C(K^{\prime
}).$ There is a closed subset $K$ of $K^{\prime}$ such that $\overline
{I}=\{a\in C(K^{\prime}):a(K)=\{0\}\}$ and $C(K^{\prime})/\overline{I}=C(K).$
Furthermore $A$ is a subalgebra of $C(K)$. In fact $C(K)$ is the completion of
$A.$ Since $\overline{I}$ is an order ideal in $\widehat{Z(E)},$ $A^{\prime
}\cong I^{0}=(\overline{I})^{0}$ is a band in $Z(E)^{\prime}.$ Then
$A^{\prime\prime}\cong Z(E)^{\prime\prime}/I^{00}=(I^{00})^{d},$ since
$I^{00}$ is a band in $Z(E)^{\prime\prime}.$ It remains to check that the
Arens product of $A^{\prime\prime}$ is identical with the product on the
subalgebra $(I^{00})^{d}.$ Let $F,G\in Z(E)^{\prime\prime}$ and $\widehat
{F},\widehat{G}\in A^{\prime\prime}$ such that $\widehat{F}=F|_{I^{0}},$
$\widehat{G}=G|_{I^{0}}.$ Let $\{a_{\alpha}\}$ and $\{b_{\beta}\}$ be nets in
$Z(E)$ that converge to $F$ and $G$ respectively in the $\sigma(Z(E)^{\prime
\prime},Z(E)^{\prime}))$-topology. Let $[a]=a+I$ for each $a\in Z(E).$ Then,
it follows that, $\{[a_{\alpha}]\}$ and $\{[b_{\beta}]\}$ converge to
$\widehat{F}$ and $\widehat{G}$ respectively in the $\sigma(A^{\prime\prime
},A^{\prime})$-topology. Then, for $f\in I^{0},$ we have%
\[
\widehat{F}\cdot\widehat{G}(f)=\underset{\alpha}{\lim}\underset{\beta}{\lim
}f([a_{\alpha}][b_{\beta}])=\underset{\alpha}{\lim}\underset{\beta}{\lim
}f(a_{\alpha}b_{\beta})=F\cdot G(f).
\]

\end{proof}
\end{lemma}

Let $J$ be an ideal in $E^{\sim}.$ For any $F\in Z(E)^{\prime\prime},$
$m(F)|_{J}\in Z(J).$ Each operator in $Z(J)$ has a unique extension to an
operator in the ideal center of the band generated by $J.$ Since bands in
$E^{\sim}$ are projection bands, we have $Z(J)=Z(E^{\sim})|_{J}.$

Suppose $J$ is an ideal in $E^{\sim}$ that separates the points of $E.$ Let
$A$ be a unital subalgebra of $Z(E).$

\begin{definition}
\label{df2} $E$ is called cyclic with respect to $A$ for the dual pair
$\langle E,J\rangle$ if there is $u\in E_{+}$ such that $(Au)^{0}=\{0\}$ in
$J.$

\begin{lemma}
\label{L3} Let $E$ be a Riesz space and $J$ be an ideal in $E^{\sim}$ that
separates the points of $E.$ Suppose $E$ is cyclic with respect to a unital
$f$-subalgebra $A$ of $Z(E)$ for the dual pair $\langle E,J\rangle.$ Then
$m(A^{\prime\prime})|_{J}=Z(J).$
\end{lemma}
\end{definition}

\begin{proof}
The norm completion $\widehat{A}$ of $A$ is a unital closed subalgebra of
$\widehat{Z(E)}.$ Therefore $\widehat{A}=C(K)$ for some compact Hausdorff
space $K$. Furthermore $A^{\sim}=A^{\prime}=C(K)^{\prime}$ and $A^{\prime
\prime}=C(K)^{\prime\prime}=C(S)$ for some hyperstonian space $S.$ (We again
mention that the usual multiplication on $C(S)$ is the Arens extension of the
product on $C(K)$ as shown in \cite{Ar}. Also the usual $C(S)$ -module
structure of $C(K)^{\prime}$ via its ideal center is the Arens homomorphism of
the bidual of $C(K)$ onto the ideal center of $C(K)^{\prime}$ (e.g.,
Proposition \ref{p4}).) In the rest of the proof, by Lemma \ref{L1}, we
consider $A^{\prime\prime}$ as a subalgebra of $Z(E)^{\prime\prime}.$

Let $u\in E_{+}$ be a cyclic vector and $f\in J_{+}.$ Then $\mu_{u,f}\in
Z(E)_{+}^{\prime}$ and $\mu_{u,f}|_{A}=\widehat{\mu}_{u,f}\in C(K)_{+}%
^{\prime}.$ Let $P_{f}$ be the band projection of $C(K)^{\prime}$ onto the
band $B(\widehat{\mu}_{u,f})$ generated by $\widehat{\mu}_{u,f}.$ By the
Lebesgue Decomposition Theorem and the Radon-Nikodym Theorem, we have%
\[
B(\widehat{\mu}_{u,f})=P_{f}\cdot C(K)^{\prime}=L^{1}(\widehat{\mu}%
_{u,f})=\{\mu\in C(K)^{\prime}:|\mu|<<\widehat{\mu}_{u,f}\}.
\]
The first equality above follows by Proposition \ref{p4}.

Suppose $e\circ f=0$ for some idempotent $e\in A^{\prime\prime}=C(S).$ Let
$\{a_{\alpha}\}$ be a net in $A$ that converges to $e$ in the $\sigma
(A^{\prime\prime},A^{\prime})$-topology. Then for each $a\in A,$%
\begin{align*}
e\cdot\widehat{\mu}_{u,f}(a)  &  =e(\widehat{\mu}_{u,f}\cdot a)=\underset
{\alpha}{\lim}\widehat{\mu}_{u,f}\cdot a(a_{\alpha})=\underset{\alpha}{\lim
}\widehat{\mu}_{u,f}(aa_{\alpha})=\underset{\alpha}{\lim}\mu_{u,f}(a_{\alpha
}a)\\
&  =\underset{\alpha}{\lim}f(a_{\alpha}au)=\underset{\alpha}{\lim}\mu
_{au,f}(a_{\alpha})=e(\mu_{au,f})=e\circ f(au)=\widehat{\mu}_{u,e\circ
f}(a)=0.
\end{align*}

Hence $0\leq e\leq1-P_{f}$. Conversely, since $u$ is a cyclic vector,
$\widehat{\mu}_{u,(1-P_{f})\circ f}=(1-P_{f})\cdot\widehat{\mu}_{u,f}=0$
implies that $(1-P_{f})\circ f=0.$ Therefore $1-P_{f}=\sup\{e\in C(S):e\circ
f=0$ and $e=e^{2}\}.$ That is, $a\circ f=0$ for some $a\in A^{\prime\prime}$
if and only if $P_{f}\cdot a=0.$

Let $T\in Z(E^{\sim})$ with $0\leq T\leq1.$ Let $f\in J_{+}.$ Then
$0\leq\widehat{\mu}_{u,Tf}\leq\widehat{\mu}_{u,f}$ in $A^{\prime}.$ Therefore,
by the Radon-Nikodym Theorem, there is $a_{f}\in L^{\infty}(\widehat{\mu
}_{u,f})\subset C(S)=A^{\prime\prime}$ such that $\widehat{\mu}_{u,Tf}%
=a_{f}\cdot\widehat{\mu}_{u,f}=\widehat{\mu}_{u,a_{f}\circ f}$. Since $u$ is a
cyclic vector it follows that $Tf=a_{f}\circ f.$ Suppose $g\in J$ such that
$0\leq g\leq f.$ $E^{\sim}$ is Dedekind complete, there is $G\in Z(E^{\sim})$
with $0\leq G\leq1$ such that $g=Gf.$ Therefore%
\[
m(a_{g})(g)=Tg=T(Gf)=G(Tf)=Gm(a_{f})(f)=m(a_{f})(Gf)=m(a_{f})(g).
\]
That is $(a_{g}-a_{f})\circ g=0.$ Therefore $(a_{g}-a_{f})\cdot P_{g}=0.$ Now
suppose $f,g\in J_{+}$ and $h=f\vee g.$ Then $(a_{f}-a_{h})\cdot P_{f}%
=(a_{g}-a_{h})\cdot P_{g}=0.$ Hence $a_{g}\cdot P_{f}=a_{f}\cdot P_{g}$ for
all $f,g\in J_{+}.$ Since $S$ is Stonian and $0\leq a_{f}\leq P_{f}\leq1$ for
all $f\in J_{+},$ there is a unique $a\in A^{\prime\prime}$ such that
$P_{f}\cdot a=a_{f}$ and $(1-\sup\{P_{f}:f\in J_{+}\})\cdot a=0.$ Then%
\[
Tf=a_{f}\circ f=(a\cdot P_{f})\circ f=a\circ(P_{f}\circ f)=a\circ f
\]
for all $f\in J_{+}.$
\end{proof}

Now we are ready to complete the proof of Proposition \ref{p5}.

\begin{proof}
(Necessity) Suppose $x\in E_{+}.$ Let $I(x)$ denote the ideal generated by $x$
in $E$ and let $Ann(x)=\{T\in Z(E):Tx=0\}.$ $Ann(x)$ is a closed order and
algebra ideal in $Z(E).$ Consider the map from $Z(E)$ into $Z(I(x))$ defined
by $T\rightarrow T|_{I(x)}.$ Clearly the map is a norm reducing positive
algebra homomorphism. Since the kernel of this map is equal to $Ann(x),$ the
map is also a lattice homomorphism. It induces a norm reducing lattice and
algebra homomorphism of $Z(E)/Ann(x)$ into $Z(I(x))$ where
$T+Ann(x)=[T]\rightarrow T|_{I(x)}.$ The induced map is an isometry. For
example, if $T\in Z(E)_{+}$ with norm $||T|_{I(x)}||$ in $Z(I(x)),$ then
$(T\wedge||T|_{I(x)}||1)|_{I(x)}=T|_{I(x)}$ and $||[T]||\leq||T|_{I(x)}||.$
Hence let $A=Z(E)|_{I(x)}$ be the normed unital $f$-subalgebra of $Z(I(x)).$
Then since $A\cong Z(E)/Ann(x)$ (isometric, lattice and algebra homomorphism),
we may think of $A^{\prime\prime}$ as a subalgebra of $Z(E)^{\prime\prime}$
(c.f., Lemma \ref{L2}). Equivalently, we may think of $A^{\prime\prime}$ as a
subalgebra of $Z(I(x))^{\prime\prime}$ (c.f., Lemma \ref{L1}).

Let $J=\{f|_{I(x)}:f\in E^{\sim}\}.$ Clearly $J$ is an ideal in $I(x)^{\sim}$
that separates the points of $I(x).$ Let $P_{x}\in Z(E^{\sim})$ denote the
band projection of $E^{\sim}$ onto $(I(x)^{0})^{d}.$ The map $f|_{I(x)}%
\rightarrow P_{x}(f)$ is a lattice homomorphism of $J$ onto the band
$(I(x)^{0})^{d}.$ Let $m_{x}$ be the Arens homomorphism of $Z(I(x))^{\prime
\prime}$ into $Z(I(x)^{\sim})$ and let $m$ be the same for $Z(E)^{\prime
\prime}$ into $Z(E^{\sim}).$ We claim that on elements of $A^{\prime\prime},$
$m_{x}$ restricted to $J$ agrees with $m$ restricted to $(I(x)^{0})^{d}.$
Namely, let $F\in A^{\prime\prime},$ $f\in E^{\sim}$ and $y\in I(x).$ Choose a
net $\{a_{\alpha}\}$ in $A$ that converges to $F$ in the $\sigma
(A^{\prime\prime},A^{\prime})$-topology. Then%
\begin{align*}
m_{x}(F)(f|_{I(x)})(y)  &  =F\circ f|_{I(x)}(y)=F(\mu_{y,f|_{I(x)}}%
)=F([\mu_{y,f|_{I(x)}}])\\
&  =\underset{\alpha}{\lim}[\mu_{y,f|_{I(x)}}](a_{\alpha})=\underset{\alpha
}{\lim}\mu_{y,f|_{I(x)}}(a_{\alpha})=\underset{\alpha}{\lim}f(a_{\alpha}y)
\end{align*}
where $[\mu_{y,f|_{I(x)}}]=\mu_{y,f|_{I(x)}}+A^{0}\in A^{\prime}%
=Z(I(x))^{\prime}/A^{0}.$ On the other hand
\begin{align*}
m(F)(P_{x}(f))(y)  &  =P_{x}(m(F)(f))(y)=m(F)(f)(y)=F\circ f(y)\\
&  =F(\mu_{y,f})=\underset{\alpha}{\lim}\mu_{y,f}(a_{\alpha})=\underset
{\alpha}{\lim}f(a_{\alpha}y).
\end{align*}
In the last string of equalities, the first equality follows because
$Z(E^{\sim})$ is commutative. The second equality follows because the range of
$1-P_{x}$ is $I(x)^{0}$ in $E^{\sim}.$ Finally the fifth equality follows
because $\mu_{y,f}\in Ann(x)^{0}=A^{\prime}$ in $Z(E)^{\prime}.$ Hence, the
claim is verified. In what follows we will keep the notation that we
established in this initial part of the proof.

Suppose that $E$ has topologically full center. This means that $I(x)$ is
cyclic with respect to the unital $f$-subalgebra $A$ of $Z(I(x))$ for the
duality $\langle I(x),J\rangle.$ Given $T\in Z(E^{\sim})$ with $0\leq T\leq1,$
let $\widehat{T}=T|_{(I(x)^{0})^{d}}\in Z(J)$ where $\widehat{T}%
(f|_{I(x)})(y)=T(P_{x}(f))(y)$ for each $f\in E^{\sim}$ and $y\in I(x).$ Then
by Lemma \ref{L3} there is $a_{x}\in A^{\prime\prime}$ with $0\leq a_{x}\leq1$
such that $m_{x}(a_{x})=\widehat{T}$ on $J.$ Then we have that%
\[
m(a_{x})(P_{x}(f))(y)=m_{x}(a_{x})(f|_{I(x)})(y)=\widehat{T}(f|_{I(x)}%
)(y)=T(P_{x}(f))(y)
\]
for all $f\in E^{\sim}$ and $y\in I(x).$ Since the range of $1-P_{x}$ is
$I(x)^{0}$ in $E^{\sim},$ we have that%
\[
m(a_{x})(f)(y)=T(f)(y)
\]
for all $f\in E^{\sim}$ and $y\in I(x).$

Let $1-e_{x}=\sup\{e\in Z(E)^{\prime\prime}:e=e^{2},$ $m(e)(f)(x)=0$ for all
$f\in E^{\sim}\}.$ Since $m$ is order continuous, we have $m(1-e_{x})(f)(x)=0$
for all $f\in E^{\sim}.$ Hence $F\circ f(x)=0$ for all $f\in E^{\sim}$ for
some $F\in Z(E)^{\prime\prime}$ if and only if $e_{x}\cdot F=0$. ( Note that
$\{\mu_{x,f}:f\in E^{\sim}\}$ is an ideal in $Z(E)^{\prime}$ and $1-e_{x}$ is
the band projection of $Z(E)^{\prime\prime}$ onto the band in $Z(E)^{\prime
\prime}$ that annihilates this ideal.) Now repeating the argument in the proof
of Lemma \ref{L3}, we find a unique $a\in Z(E)^{\prime\prime}$ such that
$0\leq a\leq1$, and $e_{x}\cdot a=a_{x}$ for each $x\in E_{+}.$ Then, for each
$f\in E^{\sim}$ and $x\in E_{+},$ we have
\[
T(f)(x)=m(a_{x})(f)(x)=m(a\cdot e_{x})(f)(x)=m(e_{x})m(a)(f)(x)=m(a)(f)(x).
\]
Therefore $m(a)=T.$
\end{proof}

We will conclude this paper by stating some immediate consequences of
Proposition \ref{p5}.

\begin{corollary}
\label{c3} Let $A$ be an $f$-algebra with point separating order dual such
that $(A^{\sim})_{n}^{\sim}$ has a unit. Then $A$ has topologically full center.
\end{corollary}

\begin{proof}
By Theorem 5.2 \cite{HdP1}, the homomorphism $\upsilon$ of $(A^{\sim}%
)_{n}^{\sim}$ is onto $Orth(A^{\sim}).$ Then Proposition \ref{p2} implies that
the Arens homomorphism $\gamma:(Orth(A)^{\sim})_{n}^{\sim}\rightarrow
Orth(A^{\sim})$ is also onto. Hence, by Corollary \ref{c2}, $m$ is onto
$Z(A^{\sim}).$ Therefore, by Proposition \ref{p5}, $A$ has topologically full center.
\end{proof}

\begin{remark}
Characterizations of $f$-algebras $A$ such that $(A^{\sim})_{n}^{\sim}$ has
unit are given in \cite{HdP1},\cite{BJ},\cite{J}. Related to Corollary
\ref{c3} , we mention that we do not know any examples of semi-prime
$f$-algebras that do not have topologically full center.
\end{remark}

\begin{corollary}
\label{c4} Let $E$ be a Riesz space with topologically full center. Suppose
$T$ is an order bounded operator on $E.$ Then $T$ commutes with $Z(E)$ if and
only if $T$ is in $Orth(E).$
\end{corollary}

\begin{proof}
$T$ is order bounded implies $T^{\prime}:E^{\sim}\rightarrow E^{\sim}.$ If $T$
commutes with $Z(E)$, then $T^{\prime}$ commutes with $m(Z(E)^{\prime\prime
}).$ Since $Z(E)$ is topologically full, $m(Z(E)^{\prime\prime})=Z(E^{\sim}).$
Therefore $T^{\prime}$ commutes with $Z(E^{\sim}).$ That is, $T^{\prime}$
commutes with the band projections on $E^{\sim}.$ Since $E^{\sim}$ is Dedekind
complete, each band in $E^{\sim}$ is a projection band. So $T^{\prime}$ is
band preserving and therefore $T^{\prime}\in Orth(E^{\sim}).$ By a result in
\cite[Theorem 3.3]{W5}, $T\in Orth(E).$
\end{proof}

In view of Examples \ref{ex4} and \ref{ex5}, the corollary may fail if $Z(E)$
is not topologically full. On the other hand, the result may be true even when
$Z(E)$ is not topologically full. The example constructed by Wickstead in
\cite{W2} shows this. We refer the reader to the introduction for more
detailed information. Corollary \ref{c4} shows that $Orth(E)$ is maximal
abelian when $Z(E)$ is topologically full. On the other hand, Wickstead's
example in \cite{W2} shows that the converse is not true.

\begin{corollary}
\label{c5} Let $E$ be a Riesz space with topologically full center. Then the
following are equivalent:\newline(1) $E^{\sim\sim}=(E^{\sim})_{n}^{\sim}%
.$\newline(2) $m$ is continuous when its domain has the $\sigma(Z(E)^{\prime
\prime},Z(E)^{\prime})$-topology and its range has the $\sigma(E^{\sim
},E^{\sim\sim})$-operator topology.
\end{corollary}

We leave the straightforward proof of the corollary to the interested reader.

Before stating our final corollary, we want to discuss its content and fix
some notation. Let $K$ be a hyperstonian space. That is $K$ is a Stonian
compact Hausdorff space and $C(K)$ is a dual Banach space. Let $C(K)_{\ast}$
denote the predual of $C(K).$ Recall that $C(K)_{\ast}=C(K)_{n}^{\prime}$, the
order continuous linear functionals on $C(K).$ Hence $C(K)_{\ast}$ is a band
in the Dedekind complete Banach lattice $C(K)^{\prime}.$ Since $Z(C(K)^{\prime
})=C(K)^{\prime\prime},$ there is an idempotent $p\in C(K)^{\prime\prime}$
such that $p$ is the band projection on $C(K)^{\prime}$ with range
$C(K)_{\ast}$. That is
\[
p\cdot C(K)^{\prime}=C(K)_{n}^{\prime}=C(K)_{\ast}.
\]
Let $E$ be a Riesz space. Its order dual $E^{\sim}$ is a Dedekind complete
Riesz space. Therefore $E^{\sim}$ has a topologically full center $Z(E^{\sim
}).$ Furthermore $Z(E^{\sim})$ is itself Dedekind complete as a Banach
lattice. In fact, it is familiar that $Z(E^{\sim})=C(K)$ for some hyperstonian
space $K$.(This will become clear in the proof of the corollary.) Let
$m:Z(E^{\sim})^{\prime\prime}\rightarrow Z(E^{\sim\sim})$ be the Arens
homomorphism of the bidual of $Z(E^{\sim})$. Since $Z(E^{\sim})$ is
topologically full, we have $m(Z(E^{\sim})^{\prime\prime})=Z(E^{\sim\sim})$
and $Ker(m)=(1-\pi)\cdot Z(E^{\sim})^{\prime\prime}$ for some idempotent
$\pi\in Z(E^{\sim})^{\prime\prime}=C(K)^{\prime\prime}$ (Proposition
\ref{p5}). We will show that $p\circ E^{\sim\sim}=(E^{\sim})_{n}^{\sim}$.

\begin{corollary}
\label{c6} Let $E$ be a Riesz space with point separating order dual $E^{\sim
}.$ Let $m$ be the Arens homomorphism of the bidual of $Z(E^{\sim})$ in
$Z(E^{\sim\sim}).$ Then

\begin{enumerate}
\item $Z(E^{\sim})$ is topologically full and $Z(E^{\sim})=C(K)$ for some
hyperstonian space $K.$

\item There is an idempotent $\pi\in C(K)^{\prime\prime}$ such that
\[
Z(E^{\sim\sim})=\pi\cdot C(K)^{\prime\prime}\text{\thinspace\thinspace
and\thinspace}\,Ker(m)=(1-\pi)\cdot C(K)^{\prime\prime}.
\]

\item There is an idempotent $p\in C(K)^{\prime\prime}$ with $p\leq\pi$ such
that
\[
p\cdot C(K)^{\prime}=C(K)_{\ast}=C(K)_{n}^{\prime}\text{ \thinspace and
\thinspace}p\circ E^{\sim\sim}=(E^{\sim})_{n}^{\sim}.
\]

\item $E^{\sim\sim}=(E^{\sim})_{n}^{\sim}$ if and only if $p=\pi.$
\end{enumerate}
\end{corollary}

\begin{proof}
1. Since $E^{\sim}$ is Dedekind complete, $Z(E^{\sim})=C(K)$ is topologically
full and $K$ is a Stonian compact Hausdorff space. (It is well known that $K$
is hyperstonian. We include a proof for the sake of completeness.) To show
that $K$ is hyperstonian, it is sufficient to see that the order continuous
linear functionals on $C(K)$ separate the points of $C(K)$ \cite{S}. Consider
$E\subset(E^{\sim})_{n}^{\sim}\subset E^{\sim\sim}.$ Take positive elements
$x\in E,$ $f\in E^{\sim}$ and $a_{\tau},a\in C(K)$ such that $\{a_{\tau}\}$ is
an increasing net with $\sup a_{\tau}=a$ in $C(K).$ Then, since $E^{\sim}$ is
Dedekind complete, $\sup a_{\tau}f=af$ in $E^{\sim}.$ Therefore $a_{\tau
}f(x)\uparrow af(x),$ since $x$ is an order continuous linear functional on
$E^{\sim}.$ Consider $\mu_{f,x}\in C(K)^{\prime}$ in the definition process of
$m,$ we have $\mu_{f,x}(b)=x(bf)=bf(x)$ for all $b\in C(K).$ Hence it follows
that $\mu_{f,x}\in C(K)_{n}^{\prime}$ for all $x\in E$ and $f\in E^{\sim}.$
Also it is clear that these linear functionals separate the points of the
center $Z(E^{\sim})=C(K).$ Therefore $K$ is hyperstonian and $C(K)_{\ast
}=C(K)_{n}^{\prime}$ is the predual of $C(K).$

2. The existence of $\pi$ is clear from Proposition \ref{p5}. An equivalent
means of defining $\pi\in C(K)^{\prime\prime}$ is by observing that $\pi$ is
the supremum of the band projections on $C(K)^{\prime}$ obtained by
considering the bands generated by each linear functional of the form
$\mu_{f,x^{\prime\prime}}\in C(K)^{\prime}$ when $f\in E^{\sim}$ and
$x^{\prime\prime}\in E^{\sim\sim}.$

3. Let $p\in C(K)^{\prime\prime}=Z(C(K)^{\prime})$ be the band projection onto
the band $C(K)_{n}^{\prime}=C(K)_{\ast}.$ An equivalent means of defining $p$
would be to observe that $p$ is the supremum of the band projections on
$C(K)^{\prime}$ obtained by considering the bands generated by each linear
functional of the form $\mu_{f,x}\in C(K)^{\prime}$ when $f\in E^{\sim}$ and
$x\in E\subset E^{\sim\sim}.$ Hence we have $p\leq\pi.$ It remains to show
that $p\circ E^{\sim\sim}=(E^{\sim})_{n}^{\sim}.$ Note that for each $f\in
E^{\sim}$ and each $x^{\prime\prime}\in E^{\sim\sim},$ we have $\mu_{f,p\circ
x^{\prime\prime}}=p\cdot\mu_{f,x^{\prime\prime}}$.( Namely, let $\{a_{\alpha
}\}$ be a net in $C(K)$ that converges to $p$ in $\sigma(C(K)^{\prime\prime
},C(K)^{\prime})$-topology. Then%
\[
\mu_{f,p\circ x^{\prime\prime}}(a)=p\circ x^{\prime\prime}(af)=p(\mu
_{af,x^{\prime\prime}})=\underset{\alpha}{\lim}\mu_{af,x^{\prime\prime}%
}\left(  a_{\alpha}\right)  =\underset{\alpha}{\lim}x^{\prime\prime}%
(a_{\alpha}af)
\]
and%
\[
p\cdot\mu_{f,x^{\prime\prime}}(a)=p(\mu_{f,x^{\prime\prime}}\cdot
a)=\underset{\alpha}{\lim}\mu_{f,x^{\prime\prime}}\cdot a(a_{\alpha
})=\underset{\alpha}{\lim}\mu_{f,x^{\prime\prime}}(aa_{\alpha})=\underset
{\alpha}{\lim}x^{\prime\prime}(a_{\alpha}af)
\]
for each $a\in C(K)$. Here the second set of displayed equalities follow from
the definition of the Arens product on the bidual of $C(K)$ when $C(K)$ is
considered as a unital $f$-algebra \cite{HdP1}.) But $p\cdot\mu_{f,x^{\prime
\prime}}\in C(K)_{\ast}=C(K)_{n}^{\prime}$ for each $f\in E^{\sim}.$ By
reversing the process we used in part (1), it follows that $p\circ
x^{\prime\prime}\in(E^{\sim})_{n}^{\sim}$. Conversely if $x^{\prime\prime}%
\in(E^{\sim})_{n}^{\sim},$ the process we used in part (1) shows that
$\mu_{f,x^{\prime\prime}}\in C(K)_{\ast}$. Therefore
\[
\mu_{f,x^{\prime\prime}}=p\cdot\mu_{f,x^{\prime\prime}}=\mu_{f,p\circ
x^{\prime\prime}}%
\]
for each $f\in E^{\sim}.$ That is, $p\circ x^{\prime\prime}=x^{\prime\prime}$
for all $x^{\prime\prime}\in(E^{\sim})_{n}^{\sim}.$ So $p\circ E^{\sim\sim
}=(E^{\sim})_{n}^{\sim}.$

Now (4) is clear from parts (2) and (3).
\end{proof}

\begin{remark}
The article \cite{HdP1} has initiated considerable research on Arens product
on the biduals of lattice ordered algebras, we include a partial list
\cite{BH}, \cite{Gr}, \cite{H}, \cite{Sd1}, \cite{Sd2}.
\end{remark}

\end{document}